\documentclass[10pt]{amsart}
\title[Hodge symmetry and decomposition   on non-K\"ahler solvmanifolds]
{Hodge symmetry and decomposition on non-K\"ahler solvmanifolds}

\author{Hisashi Kasuya}

\usepackage{amssymb}
\usepackage{amsmath}
\usepackage{amscd}
\usepackage{amstext}
\usepackage{amsfonts}
\usepackage[all]{xy}
\usepackage{multicol}

\theoremstyle{plain}
\newtheorem{theorem}{Theorem}[section] 
\theoremstyle{remark}
\newtheorem{remark}{Remark}
\theoremstyle{lemma}
\newtheorem{lemma}[theorem]{Lemma}
\theoremstyle{assumption}
\newtheorem{assumption}[theorem]{Assumption}
\theoremstyle{condition}
\newtheorem{condition}[theorem]{Condition}
\theoremstyle{definition}

\theoremstyle{proposition}
\newtheorem{proposition}[theorem]{Proposition}
\theoremstyle{corollary}
\newtheorem{corollary}[theorem]{Corollary}
\theoremstyle{remark}
\newtheorem{example}{Example}
\theoremstyle{assumption}
\newtheorem{Assumption}[theorem]{Assumption} 

\address[H.kasuya]{Graduate school of mathematical science university of tokyo japan }
\curraddr{}
\email{khsc@ms.u-tokyo.ac.jp}


\newcommand{\C}{\mathbb{C}}
\newcommand{\R}{\mathbb{R}}

\newcommand{\Z}{\mathbb{Z}}
\newcommand{\g}{\frak{g}}
\newcommand{\n}{\frak{n}}

\begin{document} 

\maketitle
\begin{abstract}
Let  $G=\C^{n}\ltimes_{\phi} \C^{m}$ with a semi-simple action $\phi: \C^{n}\to GL_{m}(\C)$ (not necessarily holomorphic).
Suppose $G$ has a lattice $\Gamma$.
Then we show that 
in  some conditions on  $G$ and $\Gamma$, $G/\Gamma$ admits a Hermitian metric such that the space of harmonic forms satisfies the Hodge symmetry and decomposition.
By this result  we give many examples of non-K\"ahler  Hermitian  solvmanifolds  satisfying the Hodge symmetry and decomposition.
\end{abstract}
\section{Introduction}
Let $M$ be a compact complex manifold with a Hermitian metric $g$.
We consider the complex-valued de Rham complex $(A^{\ast}(M),d)$, the Dolbeault complex $(A^{\ast,\ast}(M),\bar\partial)$, the space ${\mathcal H}^{\ast}_{d}(M,g)$ of complex-valued  $d$-harmonic forms and  the space ${\mathcal H}^{\ast,\ast}_{\bar \partial }(M,g)$ of complex-valued  $\bar \partial$-harmonic forms.
For the de Rham cohomology $H^{\ast}(M)$ and Dolbeault cohomology $H^{\ast,\ast}(M)$,  we have isomorphisms $H^{\ast}(M)\cong {\mathcal H}^{\ast}_{d}(M,g)$ and $H^{\ast,\ast}(M)\cong {\mathcal H}^{\ast,\ast}_{\bar \partial }(M,g)$ of linear spaces.
If $g$ is a K\"ahler metric, then we have
\[ \overline{{\mathcal H}^{p,q}_{\bar \partial }(M,g)}= {\mathcal H}^{q,p}_{\bar \partial }(M,g)\]
and
\[\bigoplus_{p+q=r} {\mathcal H}^{p,q}_{\bar \partial }(M,g) ={\mathcal H}^{r}_{d}(M,g).\]
These conditions are called the Hodge symmetry and decomposition.
The Hodge symmetry and decomposition implies the $dd^{c}$-lemma (i.e. the equation ${\rm Im}(d)\cap{\rm Ker} (d^{c})={\rm Im}(d^{c})\cap{\rm Ker} (d)={\rm Im}(d d^{c})$ holds for the operator $d^{c}=\sqrt{-1}(\partial -\bar \partial)$ )
see  \cite[(5.21)]{DGMS}.
The $dd^{c}$-lemma is very important.
It is applied to algebraic topology.
A compact complex manifold satisfying the $dd^{c}$-lemma is formal in the sense of Sullivan see \cite{DGMS}. 

We are interested in finding non-K\"ahler  compact Hermitian manifolds satisfying  the Hodge symmetry and decomposition.
This problem is not easy.
Consider nilmanifolds (i.e. compact quotients of simply connected nilpotent Lie groups by lattices).
There are many non-K\"ahler complex nilmanifolds.
But non-toral Hermitian  nilmanifolds do not satisfy the Hodge symmetry and decomposition, since non-toral   nilmanifolds are not formal in the sense of Sullivan see \cite{H}.

In this paper we consider solvmanifolds  (i.e. compact quotients of simply connected nilpotent Lie groups by lattices).
There are non-K\"ahler solvmanifolds which are formal in the sense of Sullivan.
In \cite{Ko}, we showed that a solvmanifold $G/\Gamma$ such that $G=\R^{n}\ltimes_{\phi} \R^{m}$ with a semi-simple action $\phi$ is formal in the sense of Sullivan  (see also \cite{Kas}).
In this paper we consider a Lie group $G$ as in the following assumption.
\begin{assumption}\label{as}
 $\bullet$ $G=\C^{n}\ltimes_{\phi} \C^{m}$ with a semi-simple action $\phi: \C^{n}\to GL_{m}(\C)$ (not necessarily holomorphic).
Then we have a coordinate $z_{1},\dots ,z_{n}, w_{1},\dots ,w_{m}$ of $\C^{n}\ltimes_{\phi} \C^{m}$ such that 
\[\phi(z_{1},\dots ,z_{n})(w_{1},\dots ,w_{m})=(\alpha_{1} w_{1},\dots ,\alpha_{m} w_{m})
\]
where $\alpha_{i}$ are $C^{\infty}$-characters of $\C^{n}$.\\
$\bullet$ $G$ has a lattice $\Gamma$.
Then a lattice $\Gamma$ of  $G=\C^{n}\ltimes_{\phi} \C^{m}$ is the form $\Gamma^{\prime}\ltimes_{\phi} \Gamma^{\prime\prime}$ such that $\Gamma^{\prime}$ and $\Gamma^{\prime\prime}$ are lattices of $\C^{n}$ and $\C^{m}$ respectively and the action $\phi$ of $\Gamma^{\prime}$ preserves $\Gamma^{\prime\prime}$.

\end{assumption}
Under some condition, we show that a complex solvmanifold $G/\Gamma$ admits a Hermitian metric satisfying the Hodge symmetry and decomposition. 
By this result we give many examples of non-K\"ahler  Hermitian  solvmanifolds  satisfying the Hodge symmetry and decomposition.

\section{Dolbeault cohomology of solvmanifolds}\label{MTT}
In this section we review the result in \cite{Kd}.
Let $G$ be a Lie group as in the following assumption.
\begin{Assumption}\label{Ass}
$G$ is the semi-direct product $\C^{n}\ltimes _{\phi}N$ with a left-invariant complex structure $J=J_{\C}\oplus J_{N}$ so that:\\
(1) $N$ is a simply connected nilpotent Lie group with a left-invariant complex structure $J_{N}$.\\
Let $\frak a$ and $\n$ be the Lie algebras of $\C^{n}$ and $N$ respectively.\\
(2) For any $t\in \C^{n}$, $\phi(t)$ is a holomorphic automorphism of $(N,J_{N})$.\\
(3) $\phi$ induces a semi-simple action on the Lie algebra $\n$ of $N$.\\
(4) $G$ has a lattice $\Gamma$. (Then $\Gamma$ can be written by $\Gamma=\Gamma^{\prime}\ltimes_{\phi}\Gamma^{\prime\prime}$ such that $\Gamma^{\prime}$ and $\Gamma^{\prime\prime}$ are  lattices of $\C^{n}$ and $N$ respectively and for any $t\in \Gamma^{\prime}$ the action $\phi(t)$ preserves $\Gamma^{\prime\prime}$.) \\
(5) The inclusion $\bigwedge^{\ast,\ast}\n_{\C}^{\ast} \subset A^{\ast,\ast}(N/\Gamma^{\prime\prime})$ induces an isomorphism 
\[H^{\ast,\ast}_{\bar\partial}(\n)\cong H^{\ast,\ast}_{\bar\partial }(N/\Gamma^{\prime\prime})\]
where  $\bigwedge^{\ast,\ast} \n^{\ast}_{\C}$ is the differential bigraded algebra   of the complex valued left-invariant differential forms on the nilmanifold $N/\Gamma^{\prime\prime}$.
\end{Assumption}

Consider the decomposition $\n\otimes {\C}=\n_{1,0}\oplus \n_{0,1}$ associated with $J_{N}$.
By the condition (2), this decomposition is a direct sum of $\C^{n}$-modules.
By the condition (3) we have a basis $Y_{1},\dots ,Y_{m}$ of $\n^{1,0}$ such that the action $\phi$ on $\n_{1,0}$ is represented by
$\phi(t)={\rm diag} (\alpha_{1}(t),\dots, \alpha_{m} (t))$.
Since $Y_{j}$ is a left-invariant vector field on $N$,
the vector field $\alpha_{j}Y_{j}$ on $\C^{n}\ltimes _{\phi} N$ is  left-invariant.
Hence we have a basis $X_{1},\dots,X_{n}, \alpha_{1}Y_{1},\dots ,\alpha_{m}Y_{m}$ of $\g_{1,0}$.
Let $x_{1},\dots,x_{n}, \alpha^{-1}_{1}y_{1},\dots ,\alpha_{m}^{-1}y_{m}$ be the  basis of $\g^{\ast}_{1,0}$ which is dual to $X_{1},\dots,X_{n}, \alpha_{1}Y_{1},\dots ,\alpha_{m}Y_{m}$.
Then we have 
\[\bigwedge ^{p,q}\g^{\ast}_{\C}=\bigwedge ^{p}\langle x_{1},\dots ,x_{n}, \alpha^{-1}_{1}y_{1},\dots ,\alpha^{-1}_{m}y_{m}\rangle\otimes \bigwedge ^{q}\langle \bar x_{1},\dots ,\bar x_{n}, \bar\alpha^{-1}_{1}\bar y_{1},\dots ,\bar\alpha^{-1}_{m}\bar y_{m}\rangle.
\]

\begin{lemma}{\rm (\cite[Lemma 2.2]{Kd})}\label{charr}
Let  $\alpha:\C^{n}\to \C^{\ast}$ be a  $C^{\infty}$-character of $\C^{n}$.
There exists a unique unitary character $\beta$ such that $\alpha \beta^{-1}$ is holomorphic.
\end{lemma}

By this lemma, take the unique unitary characters $\beta_{i}$ and $\gamma_{i}$ on $\C^{n}$ such that $\alpha_{i}\beta_{i}^{-1}$ and $\bar\alpha\gamma^{-1}_{i}$ are holomorphic.

\begin{theorem}{\rm (\cite[Corollary 4.2]{Kd})}\label{CORR}
Let  $B^{\ast,\ast}_{\Gamma}\subset A^{\ast,\ast}(G/\Gamma)$ be the differential bigraded subalgebra of $A^{\ast,\ast}(G/\Gamma)$ given by
\[B^{p,q}_{\Gamma}=\left\langle x_{I}\wedge  \alpha^{-1}_{J}\beta_{J}y_{J}\wedge \bar x_{K}\wedge \bar \alpha^{-1}_{L}\gamma_{L}\bar y_{L}{\Big \vert} \begin{array}{cc}\vert I\vert+\vert J\vert=p,\, \vert K\vert+\vert L\vert=q \\ (\beta_{J}\gamma_{L})_{\vert_{\Gamma}}=1\end{array}\right\rangle.
\]
Then  the inclusion $B^{\ast,\ast}_{\Gamma}\subset A^{\ast,\ast}(G/\Gamma)$ induces a cohomology isomorphism
\[H^{\ast,\ast}_{\bar \partial}(B^{\ast,\ast}_{\Gamma})\cong H^{\ast,\ast}_{\bar \partial}(G/\Gamma).
\]

\end{theorem}

\section{Main results}
Let $G$ be a Lie group as in Assumption \ref{as}.
This assumption is a special case of Assumption \ref{Ass} such that $(N,J)$  is a complex abelian Lie group.
We have 
\[\bigwedge ^{p,q}\g^{\ast}=\bigwedge ^{p}\langle dz_{1},\dots ,dz_{n}, \alpha^{-1}_{1}dw_{1},\dots ,\alpha^{-1}_{m}dw_{m}\rangle\bigotimes \bigwedge ^{q}\langle \bar dz_{1},\dots ,\bar dz_{n}, \bar\alpha^{-1}_{1}\bar dw_{1},\dots ,\bar\alpha^{-1}_{m}\bar dw_{m}\rangle.
\]
By  Lemma \ref{charr}, take the unique unitary characters $\beta_{i}$ and $\gamma_{i}$ on $\C^{n}$ such that $\alpha_{i}\beta_{i}^{-1}$ and $\bar\alpha\gamma^{-1}_{i}$ are holomorphic.
Consider the differential bigraded subalgebra $B_{\Gamma}^{\ast,\ast}$ of the Dolbeault complex $(A^{\ast,\ast}(G/\Gamma),\bar\partial)$ given by
\[B^{p,q}_{\Gamma}=\left\langle dz_{I}\wedge \alpha^{-1}_{J} \beta_{J}dw_{J}\wedge d\bar z_{K} \wedge \bar\alpha^{-1}_{L}\gamma_{L} d\bar w_{L}{\Big \vert} \begin{array}{cc}\vert I\vert+\vert J\vert=p,\, \vert K\vert+\vert L\vert=q \\ (\beta_{J}\gamma_{L})_{\vert_{\Gamma}}=1\end{array}\right\rangle.
\]
Then by Theorem \ref{CORR}, the incluison $B_{\Gamma}^{\ast,\ast}\subset A^{\ast,\ast}(G/\Gamma)$ induces a  cohomology isomorphism.

We prove:
\begin{proposition}\label{Mt}
Let $G$ be a Lie group as in Assumption \ref{as}.
Consider   the left-invariant Hermittian metric
$g=\sum dz_{i}  d\bar z_{i}+  \sum \alpha^{-1}_{i}\bar\alpha^{-1}_{i}dw_{i}d\bar w_{i}$.
Then the space ${\mathcal H}^{\ast,\ast}_{\bar \partial}(G/\Gamma)$is equal to the space $B^{\ast,\ast}_{\Gamma}$.
\end{proposition}
\begin{proof}

Let 
\[\bar\ast_{g}A^{p,q}(G/\Gamma)\to A^{n+m-p,n+m-q}(G/\Gamma)\]
 be the $\C$-anti-linear Hodge star operator of $g$ on $A^{\ast,\ast}(G/\Gamma)$.
Then we have $\bar\ast_{g}(B^{p,q})\subset B^{p,q}$ (see \cite{Kd}).
For each multi-indices $J,L \subset \{1,\dots, m\}$, $\alpha^{-1}_{J} \beta_{J}$ and  $\bar\alpha^{-1}_{L}\gamma _{L}$ are holomorphic.
Hence we have  $\bar\partial B^{p,q}_{\Gamma}=0$ 
Since the  inclusion $B^{\ast,\ast}\subset A^{\ast,\ast}(G/\Gamma)$ induces a cohomology isomorphism, we have 
\[B^{\ast,\ast}={\mathcal H}_{\bar\partial}^{\ast,\ast}(G/\Gamma,g).\]
\end{proof}

Since $B^{\ast,\ast}$ is closed under the wedge product, we have:
\begin{corollary}\label{CFo}
Let $G$ be a Lie group as in Assumption \ref{as}.
Consider   the left-invariant Hermittian metric
$g=\sum dz_{i}  d\bar z_{i}+  \sum \alpha^{-1}_{i}\bar\alpha^{-1}_{i}dw_{i}d\bar w_{i}$.
Then the space ${\mathcal H}^{\ast,\ast}_{\bar \partial}(G/\Gamma)$ of complex-valued  $\bar \partial$-harmonic forms  is closed under the wedge product.
\end{corollary}
\begin{remark}
If a compact oriented manifold $M$ admits a  Riemaniann metric such that  all products of harmonic forms are again harmonic, we call $M$  geometrically formal.
Kotschick's nice  work in \cite{Kot} stimulates us  to consider  geometrical formality.
In \cite{Kg}, we show that a solvmanifold $G/\Gamma$ such that $G=\R^{n}\ltimes_{\phi} \R^{m}$ with a semisimple action $\phi$ is geometrically formal.
Hence Corollary \ref{CFo} is a complex analogue of such result.
\end{remark}

We consider the following conditions:
\begin{condition}\label{cond}
 For each multi-indices $J,L \subset \{1,\dots, m\}$, $(\beta_{J}\gamma_{L})_{\vert_{\Gamma}}=\bf 1$  if and only if $\alpha_{J}\bar \alpha_{L}=\bf 1$ where $\bf 1$ is the trivial character.
\end{condition}
In this condition,  we have
\begin{multline*}
B^{p,q}=\left\langle dz_{I}\wedge\alpha^{-1}_{J} \beta_{J}dw_{J}\wedge d\bar z_{K}\wedge \bar\alpha^{-1}_{L}\gamma _{L} d\bar w_{L}{\Big \vert} \begin{array}{cc}\vert I\vert+\vert J\vert=p,\, \vert K\vert+\vert L\vert=q \\ \alpha_{J}\bar\alpha_{L}=1\end{array}\right\rangle\\
=\left\langle dz_{I}\wedge  dw_{J}\wedge d\bar z_{K}\wedge d\bar w_{L}{\Big \vert} \begin{array}{cc}\vert I\vert+\vert J\vert=p,\, \vert K\vert+\vert L\vert=q \\ \alpha_{J}\bar\alpha_{L}=1\end{array}\right\rangle.
\end{multline*}
Hence we have
\begin{multline*}
\overline{B^{p,q}}=\left\langle  d\bar z_{I}\wedge d \bar w_{J}\wedge d z_{K}\wedge d  w_{L}{\Big \vert} \begin{array}{cc}\vert I\vert+\vert J\vert=p,\, \vert K\vert+\vert L\vert=q \\ \alpha_{J}\bar\alpha_{L}=1\end{array}\right\rangle\\
=\left\langle   d z_{K}\wedge d  w_{L}\wedge d\bar z_{I}\wedge d \bar w_{J}{\Big \vert} \begin{array}{cc}\vert I\vert+\vert J\vert=p,\, \vert K\vert+\vert L\vert=q \\ \bar\alpha_{J}\alpha_{L}=1\end{array}\right\rangle=B^{q,p}.
\end{multline*}
We prove:
\begin{theorem}
Let $G$ be a Lie group as in Assumption \ref{as}.
We also suppose Condition \ref{cond}.
Consider   the left-invariant Hermittian metric
$g=\sum dz_{i}  d\bar z_{i}+  \sum \alpha^{-1}_{i}\bar\alpha^{-1}_{i}dw_{i}d\bar w_{i}$.
Then we have
\[ \overline{{\mathcal H}_{\bar \partial}^{p,q}(G/\Gamma,g)}={\mathcal H}_{\bar \partial}^{q,p}(G/\Gamma,g)\]
and
\[\bigoplus_{p+q=r} {\mathcal H}_{\bar \partial}^{p,q}(G/\Gamma,g) ={\mathcal H}^{r}_{d}(G/\Gamma,g).
\]

\end{theorem}

\begin{proof}

By $\overline{B^{p,q}}=B^{q,p}$ and Proposition \ref{Mt} we have 
\[ \overline{{\mathcal H}_{\bar \partial}^{p,q}(G/\Gamma,g)}={\mathcal H}_{\bar \partial}^{q,p}(G/\Gamma,g).\]

By
\[B^{p,q}=\left\langle dz_{I}\wedge  dw_{J}\wedge d\bar z_{K}\wedge d\bar w_{L}{\Big \vert} \begin{array}{cc}\vert I\vert+\vert J\vert=p,\, \vert K\vert+\vert L\vert=q \\ (\alpha_{J}\bar\alpha_{L})_{\vert_{\Gamma}}=1\end{array}\right\rangle,\]
all the elements of $B^{\ast,\ast}$ are also $d$-harmonic.
Thus by Proposition \ref{Mt} we have
\[\bigoplus_{p+q=r} {\mathcal H}_{\bar\partial}^{p,q}(G/\Gamma,g)= \bigoplus_{p+q=r} B^{p,q}\subset {\mathcal H}^{r}_{d}(G/\Gamma,g).
\]
This implies that the Fr\"olicher spectral sequence degenerates at $E_{1}$.
Hence we have 
\[\bigoplus_{p+q=r} {\mathcal H}_{\bar \partial}^{p,q}(G/\Gamma,g) = \bigoplus_{p+q=r} B^{p,q}={\mathcal H}^{r}_{d}(G/\Gamma,g).
\]
Hence the theorem follows.

\end{proof}

\section{Examples}
In \cite{Hn}, Hasegawa  showed that  a simply connected solvable Lie group $G$ with a  lattice $\Gamma$ such that $G/\Gamma$ admits a K\"ahler structure can be written as $G=\R^{2k}\ltimes _{\phi}\C^{l}$ such that
\[\phi(t_{j})((z_{1},\dots ,z_{l}))=(e^{\sqrt{-1}\theta^{j}_{1}t_{j}}z_{1},\dots ,e^{\sqrt{-1}\theta^{j}_{l}t_{j}}z_{l}),
\]
where each $e^{\sqrt{-1}\theta^{j}_{i}}$ is a root of unity.
Hence for $G=\C^{n}\ltimes_{\phi} \C^{m}$ with a semi-simple action $\phi: \C^{n}\to GL_{m}(\C)$, a compact solvmanifold $G/\Gamma$  admits no K\"ahler metric if $\phi$ has a non-unitary eigencharacter.
In particular if $G$ is completely solvable, then a compact solvmanifold $G/\Gamma$  admits no K\"ahler metric.
Hence solvmanifolds $G/\Gamma$ such that $G$ are completely solvable Lie groups as in Assumption \ref{as} and Condition \ref{cond} are non-K\"ahler  Hermittian  solvmanifolds  satisfying the Hodge symmetry and decomposition.
We give such examples. 

\begin{example}
Let $G=\C\ltimes_{\phi} \C^{2m}$ such that 
\[\phi(x+\sqrt{-1}y)(w_{1},w_{2}\dots,w_{2m-1}, w_{2m})
=(e^{a_{1}x}w_{1},e^{-a_{1}x}w_{2},\dots, e^{a_{m}x}w_{2m-1}, e^{-a_{m}x}w_{2m})
\]
for integers $a_{i}\not=0$.
We have $\beta_{2i-1}=e^{-a_{i}\pi\sqrt{-1}y}$, $\beta_{2i}=e^{a_{i}\pi\sqrt{-1}y}$.
We can write $G=\R\times (\R\ltimes_{\phi} \C^{2m})$.
$G$ has a lattice $\Gamma=t\Z\times \Delta$ where $\Delta$ is a lattice in $\R\ltimes_{\phi} \C^{2m}$ for $t>0$.
Then $B^{p,q}$ varies for the choice $t>0$.
Consider the case $t\not=\frac{r}{s}\pi$ for any $r,s\in \Z$. 
Then $(\beta_{J}\beta_{L})_{\vert_{\Gamma}}\not=\bf 1$ for any multi-indices $J,L$ such that $\alpha_{J}\alpha_{L}\not=1$.
Hence Condition  \ref{cond}  hold. 
Thus the complex solvamanifold $G/\Gamma$  with a Hermitian metric $g=\sum dz_{1}  d\bar z_{1}+  \sum( e^{-2a_{i}x}dw_{2i-1}d\bar w_{2i-1}+e^{2a_{i}x}dw_{2i}d\bar w_{2i})$ satisfies the Hodge symmetry and decomposition
but $G/\Gamma$ does not admits a K\"ahler metric.

\end{example}

\begin{example}

Let $H=\R^{n}\ltimes_{\psi} \R^{n+1}$ such that 
\[\psi(x_{1},\dots, x_{n})(y_{1},\dots, y_{n}, y_{n+1})=(e^{x_{1}}y_{1},\dots e^{x_{n}}y_{n}, e^{-x_{1}-\dots-x_{n}}y_{n+1}).
\]
Then a lattice of $H$ is constructed by a totally real algebraic number field (see \cite{TY}).
Consider $G=\C^{n}\ltimes_{\phi}\C^{n+1}$ such that 
\[\phi(z_{1},\dots, z_{n})(w_{1},\dots, w_{n}, w_{n+1})=(e^{x_{1}}w_{1},\dots e^{x_{n}}w_{n}, e^{-x_{1}-\dots-x_{n}}w_{n+1})
\]
for complex coordinate $(z_{1}=x_{1}+\sqrt{-1}y_{1},\dots ,z_{1}=x_{n}+\sqrt{-1}y_{n},w_{1},\dots, w_{n+1}  )$.
Since we have $G=\R^{n}\times (\R^{n}\ltimes_{\psi\oplus \psi} \R^{2n+2})$, the Lie group $G$ admits a lattice $\Gamma=(c_{1}\Z\times\dots \times c_{n}\Z)\times  \Delta$ where $ \Delta$ is a lattice of $\R^{n}\ltimes_{\psi\oplus \psi} \R^{2n+2}$.
In this case for a coordinate $(z_{1}=x_{1}+\sqrt{-1}y_{1},\dots ,z_{1}=x_{n}+\sqrt{-1}y_{n},w_{1},\dots, w_{n+1}  )$ of $\C^{n}$, we have $\beta_{i}=e^{-\sqrt{-1}y}$ for $1\le i\le n$ and $\beta_{n+1}=e^{\sqrt{-1}(y_{1}+\dots +y_{n})}$.
Hence if $c_{1},\dots,c_{n}$ are not rational numbers, then  Condition  \ref{cond}  hold.
Thus the complex solvamanifold $G/\Gamma$  with a Hermitian metric 
\[g= dz_{1}  d\bar z_{1}+\dots+dz_{n}  d\bar z_{n} +   e^{-2x_{1}} dw_{1}d\bar w_{1}+\dots+  e^{-2x_{n}} dw_{n}d\bar w_{n}+e^{2x_{1}+\dots +2x_{n}}dw_{n+1}d\bar w_{n+1} \]
 satisfies the Hodge symmetry and decomposition
but $G/\Gamma$ does not admits a K\"ahler metric.
\end{example}

 {\bf  Acknowledgements.} 

The author would like to express his gratitude to   Toshitake Kohno for helpful suggestions and stimulating discussions.
This research is supported by JSPS Research Fellowships for Young Scientists.


\begin{thebibliography}{40}

\bibitem{DGMS}
P. Deligne, P. Griffiths, J. Morgan, and D. Sullivan, 
Real homotopy theory of Kahler manifolds. Invent. Math. {\bf 29} (1975), no. 3, 245--274.

\bibitem{H} K. Hasegawa, Minimal models of nilmanifolds. Proc. Amer. Math. Soc. {\bf 106} (1989), no. 1, 65--71. 

\bibitem{Hn} K. Hasegawa, A note on compact solvmanifolds with K\"ahler structures. Osaka J. Math. {\bf 43} (2006), no. 1, 131--135.
\bibitem{Kd}
H. Kasuya, Techniques of computations of Dolbeault cohomology of  solvmanifolds.    Math. Z.  {\bf 273}, (2013), 437--447.
\bibitem{Kas}
H. Kasuya, Minimal models, formality and Hard Lefschetz properties  of solvmanifolds with local systems, J. Diff. Geom., {\bf 93}, (2013), 269--298.
\bibitem{Ko}
H. Kasuya, Formality and hard Lefschetz properties of aspherical manifolds,  to appear in  Osaka J. Math.
\bibitem{Kg}
H. Kasuya, Geometrical formality of solvmanifolds and solvable Lie type geometries, arXiv:1207.2390v2, to appear in {\em RIMS K\^oky\^uroku Bessatsu}, {\em Geometry of Transformation Groups and Combinatorics}.
\bibitem{Kot}
D. Kotschick,
On products of harmonic forms. Duke Math. J. {\bf 107} (2001), no. 3, 521--531.
\bibitem{TY}
N. Tsuchiya, A. Yamakawa, Lattices of some solvable Lie groups and actions of products of affine groups. Tohoku Math. J. (2) {\bf 61} (2009), no. 3, 349--364. 
\end{thebibliography}
\end{document}